\documentclass[12pt,leqno]{amsart}
\usepackage{euscript, amssymb, amsmath, amsthm}
\usepackage{epsfig}
\usepackage{graphicx}
\usepackage{caption}
\usepackage{longtable}
\usepackage{dcolumn}
\usepackage{setspace}
\usepackage[most]{tcolorbox}
\usepackage[colorlinks=true, citecolor=blue]{hyperref}
\definecolor{webred}{rgb}{0.75,0,0}
\definecolor{webgreen}{rgb}{0,0.75,0}
\definecolor{refkey}{gray}{0.75}

\numberwithin{equation}{section}

\newtheorem{theo}{Theorem}[section]
\newtheorem{lem}{Lemma}[section]

\newtheorem{Def}[theo]{Definition}
\theoremstyle{remark}
\newtheorem{rem}{Remark}[section]

\newcommand{\ep}{\varepsilon}
\def\R{{\mathbb{R}}}

\def\d{\displaystyle}
\def\e{{\varepsilon}}

\def\a{\alpha}

\date{}


\tcbset{
    frame code={}
    center title,
    left=0pt,
    right=0pt,
    top=0pt,
    bottom=0pt,
    colback=gray!10,
    colframe=white,
    width=\dimexpr\textwidth\relax,
    enlarge left by=0mm,
    boxsep=5pt,
    arc=0pt,outer arc=0pt,
    }

\begin{document}

\title[Blow-up and lifespan estimate for the generalized Tricomi equation with mixed nonlinearities]{Blow-up and lifespan estimate for the generalized Tricomi equation with mixed nonlinearities}
\author[M. Hamouda and M.A. Hamza]{Makram Hamouda$^{1}$ and Mohamed Ali Hamza$^{1}$}
\address{$^{1}$ Basic Sciences Department, Deanship of Preparatory Year and Supporting Studies, P. O. Box 1982, Imam Abdulrahman Bin Faisal University, Dammam, KSA.}

\medskip

\email{mmhamouda@iau.edu.sa (M. Hamouda)} 
\email{mahamza@iau.edu.sa (M.A. Hamza)}

\pagestyle{plain}


\maketitle
\begin{abstract}
We study in this article the blow-up of the solution of the generalized Tricomi  equation in the presence of two mixed nonlinearities,  namely we consider
\begin{displaymath}
\d (Tr) \hspace{1cm} u_{tt}-t^{2m}\Delta u=|u_t|^p+|u|^q,
\quad \mbox{in}\ \R^N\times[0,\infty),
\end{displaymath}
with small initial data, where $m\ge0$.\\
For the problem $(Tr)$ with $m=0$, which corresponds to the uniform wave speed of propagation, it is known that the presence of mixed nonlinearities generates a new blow-up region in comparison with the case of a one nonlinearity  ($|u_t|^p$ or $|u|^q$). We show in the present work that the competition between the two nonlinearities still yields a new blow region for the Tricomi equation $(Tr)$ with $m\ge0$, and we derive an estimate of the lifespan in terms of the Tricomi parameter $m$. As an application of the method developed for the study of the equation $(Tr)$ we obtain with a different approach the same blow-up result as in \cite{Lai2020} when we consider only one time-derivative nonlinearity, namely we keep only $|u_t|^p$ in the right-hand side of  $(Tr)$.
\end{abstract}


\section{Introduction}
\par\quad

We consider the following  semilinear damped wave equation characterized by a speed of propagation of polynomial type in time, namely the well-known {\it Generalized Tricomi} wave equation with mixed nonlinearities reads as follows:
\begin{equation}
\label{T-sys}
\left\{
\begin{array}{l}
\d u_{tt}-t^{2m}\Delta u=|u_t|^p+|u|^q, 
\quad \mbox{in}\ \R^N\times[0,\infty),\\
u(x,0)=\e f(x),\ u_t(x,0)=\e g(x), \quad  x\in\R^N,
\end{array}
\right.
\end{equation}
where $m\ge0$, $\ N \ge 1$ denotes the space dimension, $\e>0$ is a sufficiently  small parameter describing the smallness of the initial data, 
and  $f,g$ are positive functions chosen in the energy space with compact support.
 Moreover, the functions $f$ and $g$ are supposed to be positive and compactly supported on  $B_{\R^N}(0,R), R>0$.

Throughout this article, we assume that $m\ge0$, $p, q>1$ and $q \le \frac{2N}{N-2}$ if $N \ge 3$.

The study of the Cauchy problem \eqref{T-sys} has been taking different considerations depending on the choice of the nonlinear terms. Being not at all exhaustive, we will recall in the following some existing works in the literature in this direction. Let us first rewrite the system \eqref{T-sys} in a more general context. More precisely, we consider the following family of semilinear wave equations:
\begin{equation}
\label{T-sys-ab}
\left\{
\begin{array}{l}
\d u_{tt}-t^{2m}\Delta u=a|u_t|^p+b|u|^q, 
\quad \mbox{in}\ \R^N\times[0,\infty),\\
u(x,0)=\e f(x),\ u_t(x,0)=\e g(x), \quad  x\in\R^N,
\end{array}
\right.
\end{equation}
where $a$ and $b$ are two nonnegative constants.

Let  
$(a,b)=(0,1)$  and $m=0$ in \eqref{T-sys-ab}. It is worth mentioning that the obtained system  was subject to several works and the related blow-up phenomenon is more or less well understood and is linked to the Strauss conjecture. More precisely, this case  involves the critical power, $q_S$,  which is the positive root of the  quadratic equation $(N-1)q^2-(N+1)q-2=0$,
and is given  by
\begin{equation}\label{qS}
q_S=q_S(N):=\frac{N+1+\sqrt{N^2+10N-7}}{2(N-1)}.
\end{equation}
Indeed,  if $q \le q_S$ then no global solution for  \eqref{T-sys-ab} does exist (under some assumptions on the initial data), and for $q > q_S$ a global solution exists for small initial data; see e.g. \cite{John2,Strauss,YZ06,Zhou}. \\

Now,  the case $m = 0$  and $(a,b)=(1,0)$ is connected with the Glassey conjecture saying that the critical power $p_G$ should be given by
\begin{equation}\label{Glassey}
p_G=p_G(N):=1+\frac{2}{N-1}.
\end{equation}
The critical value $p_G$ generates two regions for the power $p$; one stating the global existence (for $p>p_G$) and another one for the nonexistence (for $p \le p_G$) of a global small data solution; see e.g. \cite{Hidano1,Hidano2,John1,Sideris,Tzvetkov,Zhou1}.\\

The case  $m = 0$ and $a, b \neq 0$ (we may assume that $(a,b)=(1,1)$) is characterized by two mixed nonlinearities. In this case, and having   $p \le p_G$ or $q \le q_S$, the blow-up of the solution of  \eqref{T-sys-ab} is easy. However, compared to a one single nonlinearity,  the presence of mixed nonlinearities produces a supplementary blow-up region. In fact, for $p > p_G$ and $q > q_S$, the new blow-up frontier  is characterized by the following relationship between $p$ and $q$:
\begin{equation}\label{1.5-bis}
\lambda(p, q, N):=(q-1)\big[(N-1)p-2\big]< 4;
\end{equation}
see e.g. \cite{Dai,Han-Zhou,Hidano3,Wang} for more details.\\

Now, let us consider $m\ge0$ and $(a,b)=(0,1)$   in \eqref{T-sys-ab}. This case corresponds to the semilinear generalized Tricomi equation with power nonlinearity which has been widely studied in several works. In \cite{He,LT19}, the authors proved the blow-up results for both  the subcritical  and  critical cases. However,   the question of global existence for the system  \eqref{T-sys-ab} with small data in the supercritical case  has been proved only for $m=1$ and for dimensions $N=1,2$; see e.g. \cite{He-phd,He,He2,He3}.   Hence,   the critical exponent for \eqref{T-sys-ab}, that we denote by  $q_{C}(N,m)$,
should be given by the greatest root of the following quadratic equation
\begin{align}
\mathcal{Q}_{tr}(q):=((m+1)N-1)q^2-((m+1)N+1-2m)q-2(m+1)=0.
\end{align} 
Naturally, for $m=0$ we find again the well-known Strauss exponent $q_{S}(N)$, namely $q_{C}(N,0)=q_{S}(N)$, where $q_{S}(N)$ is given by \eqref{qS}.\\

For $m\ge0$ and $(a,b)=(1,0)$, the system \eqref{T-sys-ab} is given by
\begin{equation}
\label{T-sys-bis}
\left\{
\begin{array}{l}
\d u_{tt}-t^{2m}\Delta u=|u_t|^p, 
\quad \mbox{in}\ \R^N\times[0,\infty),\\
u(x,0)=\e f(x),\ u_t(x,0)=\e g(x), \quad  x\in\R^N.
\end{array}
\right.
\end{equation}
Concerning the blow-up results and lifespan estimate of the solution  of \eqref{T-sys-bis}, it was proven in \cite{LP-tricomi} that for $p \le p_G(N(m+1))$ we have the blow-up. To prove this result, the authors used,  among other tools, the integral representation formula in one space dimension. However, this result was recently improved in \cite{Lai2020}. The main tool in the latter work is the construction of adequate test functions using the properties of Bessel functions. Hence, the new obtained region involves  an almost sure candidate for the critical exponent, that is  
\begin{equation}
p \le p_{tr}(N,m):=1+\frac{2}{(m+1)(N-1)-m}.
\end{equation}
 The authors in \cite{Lai2020} conjecture that $p_{tr}(N,m)$ is in fact the critical exponent for the blow-up of the solution of \eqref{T-sys-bis}. Of course this has to be confirmed by a global existence result for \eqref{T-sys-bis}.\\

We are interested in this article in studying the blow-up result of the solution of (\ref{T-sys}) for $m\ge0$. First, let us recall that the case $m=0$ in (\ref{T-sys}) corresponds in fact to the classical wave equation with mixed nonlinearities which, compared to the case with one nonlinearity $|u_t|^p$ or $|u|^q$, for  $p > p_G(N)$ or $q > q_S(N)$, generates a new blow-up region, given by \eqref{1.5-bis}, due to the interaction between the two mixed nonlinearities. In this  work we will  examine the influence of the parameter $m$ on the blow-up result and the lifespan estimate for the Cauchy problem (\ref{T-sys}) for $m\ge0$. Naturally,  the emphasis will be here for the exponents $p > p_{tr}(N,m)$ or $q > q_C(N,m)$.  Our target is to give the curve delimiting the new blow-up region for the Tricomi equation given by  \eqref{T-sys}, and to look into its change compared to \eqref{1.5-bis}. 

The Tricomi equation constitutes somehow  a time-dependent wave speed of propagation. As shown below (see \eqref{eq-v}),  the effect of this Tricomi term is more or less similar to a scale-invariant damping. For instance, let $v(x,\tau)=u(x,t)$ where 
\begin{equation}\label{xi}
\tau=\xi(t):=\frac{t^{m+1}}{m+1}.
\end{equation} 
Hence, the solution $v(x,\tau)$ verifies the following equation:
\begin{align} \label{eq-v}
v_{\tau\tau}-\Delta v+ \frac{\mu_m}{\tau} \,\partial_\tau v = C_{m,p} \tau^{\mu_m (p-2)} |\partial_{\tau} v|^p+\mathcal{K}_{m} \tau^{-2\mu_m} |v|^q, 
\quad \mbox{in}\ \R^N\times[0,\infty),
\end{align} 
where $\mu_m := \frac{m}{m+1}$, $\mathcal{K}_{m}=(m+1)^{-2\mu_m}$ and $C_{m,p}= (m+1)^{\mu_m(p-2)}$. Note that the equation \eqref{eq-v} is somehow related to the Euler-Darboux-Poisson equation. Moreover, thanks to this transformation, which implies a kind of similarity with the scale-invariant damping case, we can inherit the methods used in our previous works \cite{Our,Our2, Our3} to build the proofs of our main results which are related, as a first target, to the blow-up of the solution of (\ref{T-sys}) and, as a secondary aim, to the blow-up of \eqref{T-sys-bis}.\\

The  article is organized as follows. The beginning of Section \ref{sec-main} is devoted to the setting of  the weak formulation of (\ref{T-sys}) in the energy space. Then, we state the main theorems of our work, and we end this section by some remarks.  To develop the proofs of our results, we prove some technical lemmas  in Section \ref{aux}. Finally, Sections \ref{proof} and \ref{sec-ut} are committed to the proofs of Theorems \ref{blowup} and \ref{th_u_t}.

\section{Main Results}\label{sec-main}
\par

In this section, we start by giving a sense to the energy solution of (\ref{T-sys}) which  reads  as follows:
\begin{Def}\label{def1}
 Let $f,g \in \mathcal{C}_0^{\infty}(\R^N)$. We say that $u$ is an energy  solution of
 (\ref{T-sys}) on $[0,T)$
if
\begin{displaymath}
\left\{
\begin{array}{l}
u\in \mathcal{C}([0,T),H^1(\R^N))\cap \mathcal{C}^1([0,T),L^2(\R^N)), \vspace{.1cm}\\
 u \in L^q_{loc}((0,T)\times \R^N) \ \text{and} \ u_t \in L^p_{loc}((0,T)\times \R^N),
 \end{array}
  \right.
\end{displaymath}
satisfies, for all $\Phi\in \mathcal{C}_0^{\infty}(\R^N\times[0,T))$ and all $t\in[0,T)$, the following equation:
\begin{equation}
\label{energysol2}
\begin{array}{l}
\d\int_{\R^N}u_t(x,t)\Phi(x,t)dx-\int_{\R^N}u_t(x,0)\Phi(x,0)dx \vspace{.2cm}\\
\d -\int_0^t  \int_{\R^N}u_t(x,s)\Phi_t(x,s)dx \,ds+\int_0^t s^{2m} \int_{\R^N}\nabla u(x,s)\cdot\nabla\Phi(x,s) dx \,ds\vspace{.2cm}\\
\d  =\int_0^t \int_{\R^N}\left\{|u_t(x,s)|^p+|u(x,s)|^q\right\}\Phi(x,s)dx \,ds.
\end{array}
\end{equation}
\end{Def}
It is clear that the weak formulation corresponding to (\ref{T-sys-bis}) is also given by an analogous identity to \eqref{energysol2} with omitting  the nonlinear term $|u|^q$.

Now, we state the main results of this article which are related to the blow-up region and the lifespan estimate of the solutions of  (\ref{T-sys}) and (\ref{T-sys-bis}), respectively. 
\begin{theo}
\label{blowup}
Let $p, q>1$ and $m \ge 0$  such that 
\begin{equation}\label{assump}
\Lambda(p, q, N,m)<4,
\end{equation}
where  $\Lambda(p, q, N,m)$ is defined by 
\begin{equation}\label{1.5}
\Lambda(p, q, N,m):=(q-1)\big[(m+1)(N-1)p-m(p-2)-2\big], 
\end{equation}
 and $p>p_{tr}(N,m)$ and $q>q_C(N,m)$.\\
Suppose that  $f,g \in \mathcal{C}_0^{\infty}(\R^N)$ are non-negative functions which are compactly supported on  $B_{\R^N}(0,R)$,
and  do not vanish everywhere.
Let $u$ be an energy solution of \eqref{T-sys} on $[0,T_\e)$ (in the sense of \eqref{energysol2}) such that $\mbox{\rm supp}(u)\ \subset\{(x,t)\in\R^N\times[0,\infty): |x|\le \xi(t)+R\}$. 
Then, there exists a constant $\e_0=\e_0(f,g,N,R,p,q,m)>0$
such that $T_\e$ verifies
\[
T_\e\leq
 C \,\e^{-\frac{2p(q-1)}{4-\Lambda(p, q, N,m)}},
\]
 where $C$ is a positive constant independent of $\e$ and $0<\e\le\e_0$.
\end{theo}

\begin{theo}
\label{th_u_t}
Let $p>1$ and $m \ge 0$. Assume that the initial data $f$ and $g$ verify the same hypotheses as in Theorem  \ref{blowup}. Then, for  $u$ an energy solution of \eqref{T-sys-bis} on $[0,T_\e)$ which satisfies $\mbox{\rm supp}(u)\ \subset\{(x,t)\in\R^N\times[0,\infty): |x|\le \xi(t)+R\}$, there exists a constant $\e_0=\e_0(f,g,N,R,p,m)>0$
such that $T_\e$ verifies
\begin{displaymath}
T_\e \leq
\d \left\{
\begin{array}{ll}
 C \, \e^{-\frac{2(p-1)}{2-((m+1)(N-1)-m)(p-1)}}
 &
 \ \text{for} \
 1<p<p_{tr}(N,m), \vspace{.1cm}
 \\
 \exp\left(C\e^{-(p-1)}\right)
&
 \ \text{for} \ p=p_{tr}(N,m)=1+\frac{2}{(m+1)(N-1)-m},
\end{array}
\right.
\end{displaymath}
 where $C$ is a positive constant independent of $\e$ and $0<\e\le\e_0$.
\end{theo}

\begin{rem}
Note that for $m=0$ in \eqref{T-sys}, we have $\Lambda(p, q, N,0)=\lambda(p, q, N)$, which are respectively given by \eqref{1.5} and \eqref{1.5-bis}, and this gives raise to  the classical wave equation with combined nonlinearities. 
\end{rem}

\begin{rem}
Let
$$p_0:=p_{tr}(N,m)+\beta_1=1+\frac{2}{(m+1)(N-1)-m}+\beta_1,$$
and
\begin{align*}
q_0 :=1+ \frac{4}{(m+1)(N-1)p_0-m(p_0-2)-2}-\beta_2
>1+\frac{4}{(m+1)N-1}-\beta_3,
\end{align*}
where $\beta_i$ is a positive constant taken small enough for all $i=1,2,3$.\\
It is easy to see that $\mathcal{Q}_{tr}\left(1+\frac{4}{(m+1)N-1}\right)>0$. Then, for $\beta_1$ and $\beta_2$ small enough, we have $\mathcal{Q}_{tr}(q_0)>0$. Consequently, we have $q_0 > q_C(N,m)$ for $\beta_1$ and $\beta_2$ small enough. Hence,  the pair  $(p_0,q_0)$  satisfies  \eqref{assump}, $p_0 > p_{tr}(N,m)$ and $q_0 > q_C(N,m)$. Hence, the hypothesis on $p$ and $q$ in Theorem \ref{blowup} makes sense.
\end{rem}

\begin{rem}
We note  that,  for $p > p_G(N)$ and $q > q_S(N)$, the equality in \eqref{1.5-bis} yields the global existence of the solution of \eqref{T-sys-ab} (with $m = 0$ and $(a,b)=(1,1)$) without  the intermediate step of ``almost global solution''; see \cite{Hidano3}. This is in fact related to the presence of mixed  nonlinearities. Naturally, it is interesting to see whether this phenomenon still occurs for the case $m > 0$, namely when $\Lambda(p, q, N,m)=4$ and for $p>p_{tr}(N,m)$ and $q>q_C(N,m)$.   To complete the whole picture, the global existence in-time of the solution of \eqref{T-sys} will be studied in a subsequent work.
\end{rem}

\section{Auxiliary results}\label{aux}
\par

First, we start by introducing a test function  which somehow follows the dynamics of the linear equation associated with \eqref{T-sys} and that we will use subsequently to derive the behavior of the solution of \eqref{T-sys}. More precisely, we have
\begin{equation}
\label{test11}
\psi(x,t):=\rho(t)\phi(x);
\quad
\phi(x):=
\left\{
\begin{array}{ll}
\d\int_{S^{N-1}}e^{x\cdot\omega}d\omega & \mbox{for}\ N\ge2,\vspace{.2cm}\\
e^x+e^{-x} & \mbox{for}\  N=1,
\end{array}
\right.
\end{equation}
where $\phi(x)$ is introduced in \cite{YZ06},   and $\rho(t)$   verifies
\begin{equation}\label{lambda}
\d \frac{d^2 \rho(t)}{dt^2}-t^{2m}\rho(t)=0, \quad \forall \ t\ge0,
\end{equation}
The expression of  $\rho(t)$ is given by (\cite{He, Ikeda1}):
\begin{equation}\label{lmabdaK}
\rho(t)=\left\{
\begin{array}{ll}\d \a_m t^{\frac{1}{2}}K_{\frac{1}{2m+2}}\left(\frac{t^{m+1}}{m+1}\right), &\quad \forall \ t>0,\\
1, &\quad \text{at} \ t=0,
 \end{array}
  \right.
\end{equation}
where 
$$K_{\nu}(t)=\int_0^\infty\exp(-t\cosh \zeta)\cosh(\nu \zeta)d\zeta,\ \nu\in \mathbb{R},$$
and $$\d \a_m=\frac{2}{(2m+2)^{\frac{1}{2m+2}} \Gamma(\frac{1}{2m+2})}.$$
Thanks to some properties of  the modified Bessel function of second kind $K_{\nu}(t)$ near $t=0$, we infer that   $\rho(t)$ satisfies  
\begin{equation}\label{rho'(0)}
\rho'(0)=-(2m+2)^{\frac{m}{m+1}}\frac{\left(\Gamma \left(\frac{2m+1}{2m+2}\right)\right)^{\frac{2m+1}{2m+2}}}{\Gamma \left(\frac{1}{2m+2}\right)}.
\end{equation}
On the other hand, it is easy to see that $\phi(x)$ satisfies
\begin{equation}\label{phi-pp}
\Delta\phi=\phi.
\end{equation}
Consequently,  the function $\psi(x,t)=\phi(x)  \rho(t)$ is solution of the following equation:
\begin{equation}\label{lambda-eq}
\partial^2_t \psi(x, t)-t^{2m}\Delta \psi(x, t) =0.
\end{equation}

For a later use, we list here some properties of the function $\rho(t)$ (\cite{Erdelyi, Gaunt,Ikeda1,Tu-Lin}):
\begin{itemize}
\item[{\bf (i)}] The functions $\rho(t)$ and $-\rho'(t)$ are decreasing on $[0,\infty)$ and verify $\lim_{t \to +\infty}\rho(t)=\lim_{t \to +\infty}\rho'(t)=0$.
\item[{\bf (ii)}] For all $t>1$, there exist constants $C_1$ and $C_2$ such that the function  $\rho(t)$ verifies 
\begin{equation}\label{est-rho}
C_1^{-1}t^{-\frac{m}{2}} \exp(-\xi(t)) \le  \rho(t) \le C_1 t^{-\frac{m}{2}} \exp(-\xi(t)), (\xi(t) \ \text{is given by} \ \eqref{xi}),
\end{equation}
and  
\begin{equation}\label{est-rho'}
C_2^{-1} t^{\frac{m}{2}}\exp(-\xi(t)) \le|\rho'(t)| \le C_2 t^{\frac{m}{2}}\exp(-\xi(t)).
\end{equation}
\item[{\bf (iii)}] We have
\begin{equation}\label{lambda'lambda1}
\d \lim_{t \to +\infty} \left(\frac{\rho'(t)}{t^m \rho(t)}\right)=-1.
\end{equation}
\end{itemize}
Note that the proof of the properties {\bf (i)} and {\bf (ii)} can be found in \cite{Gaunt}, and the one of {\bf (iii)} is shown in the appendix (Section \ref{app}).\\


Along this work, we will denote by $C$  any positive constant which may depend on the data ($p,q,m,N,R,f,g$) but not on $\ep$ and whose  value may change from line to line. However,  in some occurrences, we will make precise the dependence of the constant $C$ on the parameters of the problem.\\

The following lemma  holds true for the function $\psi(x, t)$.
\begin{lem}[\cite{YZ06}]
\label{lem1} Let  $r>1$.
There exists a constant $C=C(N,R,p,r)>0$ such that
\begin{equation}
\label{psi}
\int_{|x|\leq \xi(t)+R}\Big(\psi(x,t)\Big)^{r}dx
\leq C\rho^r(t)e^{r\xi(t)}(1+\xi(t))^{\frac{(2-r)(N-1)}{2}},
\quad\forall \ t\ge0.
\end{equation}
\end{lem}
\par
We now introduce the following functionals:
\begin{equation}
\label{F1def}
G_1(t):=\int_{\R^N}u(x, t)\psi(x, t)dx,
\end{equation}
and
\begin{equation}
\label{F2def}
G_2(t):=\int_{\R^N}u_t(x,t)\psi(x, t)dx.
\end{equation}
The next two lemmas give the first  lower bounds for $G_1(t)$ and $G_2(t)$, respectively. More precisely, we will prove that $t^mG_1(t)$ and $G_2(t)$ are two {\it coercive} functions for $t$ large enough. 

We recall that the proof of Lemma \ref{F1} is known in the literature; see \cite{He}. However, in order to make the presentation complete, we  include here all the details about the proof of this lemma. Nevertheless, Lemma \ref{F11} constitutes a novelty in this work and its utilization in the proofs of Theorems \ref{blowup}  and  \ref{th_u_t} is crucial.  
\begin{lem}
\label{F1}
Let $u$ be an energy solution of the system \eqref{T-sys}  with initial data satisfying the assumptions in Theorem \ref{blowup}. Then, there exists $T_0=T_0(m)>1$ such that 
\begin{equation}
\label{F1postive}
G_1(t)\ge C_{G_1}\, \e \, t^{-m}, 
\quad\text{for all}\ t \ge T_0,
\end{equation}
where $C_{G_1}$ is a positive constant which may depend on $f$, $g$, $N,R$ and $m$.
\end{lem}
\begin{proof} 
Let $ t \in (0,T)$.  Using Definition \ref{def1},  performing an integration by parts in space in the fourth term in the left-hand side of \eqref{energysol2} and then substituting  $\Phi(x, t)$ by $\psi(x, t)$, we obtain
\begin{equation}
\begin{array}{l}\label{eq4}
\d \int_{\R^N}u_t(x,t)\psi(x,t)dx
-\e\int_{\R^N}g(x)\psi(x,0)dx \vspace{.2cm}\\
\d-\int_{0}^{t}\int_{\R^N}\left\{
u_t(x,s)\psi_t(x,s)+s^{2m}u(x,s)\Delta\psi(x,s)\right\}dx \, ds 
=\int_{0}^{t}\mathcal{N}(s) \, ds,
\end{array}
\end{equation}
where
\begin{equation}
\label{N-exp}
\d \mathcal{N}(t)=\int_{\R^N}\left\{|u_t(x,t)|^p+|u(x,t)|^q\right\}\psi(x,t)dx.
\end{equation}
Performing integration by parts for the first and third terms in the second line of \eqref{eq4} and utilizing \eqref{test11} and \eqref{lambda-eq}, we infer that
\begin{equation}
\begin{array}{l}\label{eq5-12}
\d \int_{\R^N}\big[u_t(x,t)\psi(x,t)- u(x,t)\psi_t(x,t)\big] dx =
\int_{0}^{t}\mathcal{N}(s) \, ds + \e C_m(f,g),
\end{array}
\end{equation}
where 
\begin{equation}\label{Cfg}
C_m(f,g):=\int_{\R^N}\big[a(m)f(x)+g(x)\big]\phi(x)dx, 
\end{equation}
with $a(m):=-\rho'(0)$; and $\rho'(0)$ is given by    \eqref{rho'(0)}.\\
It is clear   that the constant $C_m(f,g)$ is positive thanks to \eqref{rho'(0)} and the non negativity of the initial data which do not vanish everywhere. 
Hence, using the definition of $G_1$, as in \eqref{F1def},  and \eqref{test11}, the equation  \eqref{eq5-12} yields
\begin{equation}
\begin{array}{l}\label{eq6}
\d G_1'(t)+\Gamma(t)G_1(t)=\int_{0}^{t}\mathcal{N}(s) \, ds +\e \, C_m(f,g),
\end{array}
\end{equation}
where 
\begin{equation}\label{gamma}
\Gamma(t):=-2\frac{\rho'(t)}{\rho(t)}.
\end{equation}
Multiplying  \eqref{eq6} by $\frac{1}{\rho^2(t)}$ and integrating over $(0,t)$, we deduce  that
\begin{align}\label{est-G1}
 G_1(t)
\ge G_1(0)\rho^2(t)+{\e}C_m(f,g)\rho^2(t)\int_0^t\frac{ds}{\rho^2(s)}.
\end{align}
Thanks to \eqref{lmabdaK} and  the fact that $G_1(0)>0$, the estimate \eqref{est-G1} yields
\begin{align}\label{est-G1-1}
 G_1(t)
\ge {\e}C_m(f,g) t K^2_{\frac{1}{2m+2}}\left(\xi(t)\right)\int^t_{t/2}\frac{1}{sK^2_{\frac{1}{2m+2}}\left(\xi(s)\right)}ds, \quad \forall \ t>0,
\end{align}
where $\xi(t)$ is given by \eqref{xi}.\\
From \eqref{Kmu}, we have the existence of $T_0=T_0(m)>1$ such that 
\begin{align}\label{est-double}
\xi(t)K^2_{\frac{1}{2m+2}}(\xi(t))>\frac{\pi}{4} e^{-2\xi(t)} \quad \text{and}  \quad \xi(t)^{-1}K^{-2}_{\frac{1}{2m+2}}(\xi(t))>\frac{1}{\pi} e^{2\xi(t)}, \ \forall \ t \ge T_0/2.
\end{align}
Hence, combining \eqref{est-double} in  \eqref{est-G1-1} and using \eqref{xi},   we deduce that
\begin{align}\label{est-G1-3}
 G_1(t)
\ge \e CC_m(f,g)t^{-m}, \ \forall \ t \ge T_0.
\end{align}

This ends the proof of Lemma \ref{F1}.
\end{proof}

Now we are in a position to prove  the following lemma.
\begin{lem}\label{F11}
Let   $u$ be an energy solution of the system \eqref{T-sys} with initial data fulfilling the assumptions in Theorem \ref{blowup}. Then,   there exists $T_1=T_1(m)>0$ such that 
\begin{equation}
\label{F2postive}
G_2(t)\ge C_{G_2}\, \e, 
\quad\text{for all}\ t  \ge  T_1,
\end{equation}
where $C_{G_2}$ is a positive constant depending possibly on $f$, $g$, $N$ and $m$.
\end{lem}
 
\begin{proof}
Let $t \in [0,T)$ and recall the definitions of $G_1$ and  $G_2$, given respectively by \eqref{F1def} and  \eqref{F2def}, \eqref{test11} and the identity
 \begin{equation}\label{def23}\d G_1'(t) -\frac{\rho'(t)}{\rho(t)}G_1(t)= G_2(t),\end{equation}
 then the equation  \eqref{eq6} gives
\begin{equation}
\begin{array}{l}\label{eq5bis}
\d G_2(t)-\frac{\rho'(t)}{\rho(t)}G_1(t)
=\d \int_{0}^{t}\mathcal{N}(s) \, ds +\e \, C_m(f,g),
\end{array}
\end{equation}
where $\mathcal{N}(t)$ is given by \eqref{N-exp}.\\
Differentiating the  equation \eqref{eq5bis} in time, we get
\begin{align}\label{F1+bis}
\d G_2'(t)-\frac{\rho'(t)}{\rho(t)}G'_1(t)-\left(\frac{\rho''(t)\rho(t)-(\rho'(t))^2}{\rho^2(t)}\right)G_1(t) 
\d =\mathcal{N}(t).
\end{align}
Using  \eqref{lambda} and   \eqref{def23}, the equation \eqref{F1+bis} yields
\begin{align}\label{F1+bis2}
\d G_2'(t)-\frac{\rho'(t)}{\rho(t)}G_2(t)-t^{2m}G_1(t) =  \mathcal{N}(t).
\end{align}
Recall the definition of $\Gamma(t)$, given by \eqref{gamma}, we infer that 
\begin{equation}\label{G2+bis3}
\begin{array}{c}
\d G_2'(t)+\frac{3\Gamma(t)}{4}G_2(t)= \mathcal{N}(t)+\Sigma_1(t)+\Sigma_2(t),
\end{array}
\end{equation}
where 
\begin{equation}\label{sigma1-exp}
\Sigma_1(t)=\d -\frac{\rho'(t)}{2\rho(t)}\left(G_2(t)-\frac{\rho'(t)}{\rho(t)}G_1(t)\right),
\end{equation}
and
\begin{equation}\label{sigma2-exp}
\Sigma_2(t)=\d t^{2m}\left(1-\frac12\left(\frac{\rho'(t)}{t^{m}\rho(t)}\right)^2 \right)  G_1(t).
\end{equation}

Now, from \eqref{lambda'lambda1}, we deduce the existence of $\tilde{T}_1=\tilde{T}_1(m) \ge T_0$ such that 
\begin{equation}\label{T1tilde}
\d \frac12 \le \frac{-\rho'(t)}{t^{m}\rho(t)} \le \frac54, \quad \forall \ t > \tilde{T}_1.
\end{equation}
  Consequently, using  \eqref{eq5bis}, we deduce that
\begin{equation}\label{sigma1}
\d \Sigma_1(t) \ge \frac{\e \, C_m(f,g) \,t^{m}}{4}+\frac{t^{m}}{4}\int_0^t \mathcal{N}(s) ds, \quad \forall \ t \ge \tilde{T}_1. 
\end{equation}
Furthermore, using Lemma \ref{F1} and \eqref{T1tilde}, we have 
\begin{equation}\label{sigma2}
\d \Sigma_2(t) \ge 0, \quad \forall \ t  \ge  \tilde{T}_1. 
\end{equation}
Combining \eqref{G2+bis3}, \eqref{sigma1} and \eqref{sigma2}, we infer that
\begin{equation}\label{G2+bis4}
\begin{array}{c}
\d G_2'(t)+\frac{3\Gamma(t)}{4}G_2(t)\ge  \frac{\e \, C_m(f,g) \,t^{m}}{4}+\mathcal{N}(t)+\frac{t^{m}}{4}\int_0^t \mathcal{N}(s) ds, \quad \forall \ t  \ge  \tilde{T}_1.
\end{array}
\end{equation}
At this level, we can ignore the nonnegative nonlinear  term $\mathcal{N}(t)$\footnote{ In fact, the conservation of the term $\mathcal{N}(t)$ up to this step is only useful for proving Theorem \ref{th_u_t} later on. So, we could get rid of it earlier in the proof without any change.}.  \\
Now, multiplying  \eqref{G2+bis4} by $\frac{1}{\rho^{3/2}(t)}$ and integrating over $(\tilde{T}_1,t)$, we deduce  that
\begin{align}\label{est-G111-1}
 G_2(t)
\ge G_2(\tilde{T}_1)\frac{\rho^{3/2}(t)}{\rho^{3/2}(\tilde{T}_1)}+\frac{\e \, C_m(f,g)}{4}\rho^{3/2}(t)\int_{\tilde{T}_1}^t\frac{s^m}{\rho^{3/2}(s)}ds, \quad \forall \ t  \ge  \tilde{T}_1.
\end{align}
Now, in order to prove that $G_2(t)$ is coercive starting from a relatively large time, we  will show in the following that $G_2(t) \ge 0$ for all $t \ge 0$. For that purpose, we use \eqref{F1+bis2} and we easily obtain
\begin{align}\label{G2+}
\d \left(\frac{G_2(t)}{\rho(t)}\right)' =\rho^{-1}(t) \left( t^{2m}G_1(t)+\mathcal{N}(t)\right), \quad \forall \ t \ge 0.
\end{align}
Hence, employing the above identity  together with \eqref{est-G1} and using the fact that the initial data are nonnegative, we infer that
\begin{equation}\label{G2+bis}
G_2(t) \ge 0, \quad \forall \ t \ge 0.
\end{equation} 
Therefore, using \eqref{G2+bis} and  the positivity of $\rho(t)$, we deduce that 
\begin{align}\label{est-G1111-1}
 G_2(\tilde{T}_1)\frac{\rho^{3/2}(t)}{\rho^{3/2}(\tilde{T}_1)}
\ge 0, \quad \forall \ t \ge 0.
\end{align}
Employing  \eqref{est-G1111-1}, the estimate \eqref{est-G111-1} yields, for all $t \ge \tilde{T}_1$,
\begin{equation}\label{est-G2-12}
 G_2(t)
\ge   C\,{\e}\rho^{3/2}(t)\int_{\tilde{T}_1}^t\frac{s^m}{\rho^{3/2}(s)}ds.
\end{equation}
We rewrite \eqref{est-double} as follows:
\begin{align}\label{est-double-1}
(\xi(t))^\frac{3}{4}K^\frac{3}{2}_{\frac{1}{2m+2}}(\xi(t))>\left(\frac{\pi}{4}\right)^\frac{3}{4} e^{-\frac{3}{2}\xi(t)} \quad \text{and}  \quad (\xi(t))^\frac{-3}{4}K^\frac{-3}{2}_{\frac{1}{2m+2}}(\xi(t))>\left(\frac{1}{\pi}\right)^\frac{3}{4} e^{\frac{3}{2}\xi(t)}, \ \forall \ t \ge T_0/2.
\end{align}
Hence, using the definitions of the expression of $\rho(t)$ and $\xi(t)$, given respectively by \eqref{lmabdaK} and \eqref{xi}, we have
\begin{align}\label{est-G1-2}
 G_2(t)
&\ge \e Ct^{\frac{-3m}{4}}e^{-\frac{3}{2}\xi(t)}\int^t_{t/2}s^\frac{3m}{4}\xi'(s)e^{\frac{3}{2}\xi(s)}ds\\
&\ge  \e Ce^{-\frac{3}{2}\xi(t)}\int^t_{t/2}\xi'(s)e^{\frac{3}{2}\xi(s)}ds, \ \forall \ t \ge 2\tilde{T}_1.\nonumber
\end{align}
Finally, using $e^{\frac32\xi(t)}>2e^{\frac32\xi(t/2)}, \forall \ t \ge T_1:=T_1(m)\ge2\tilde{T}_1$, we deduce that
\begin{align}\label{est-G1-2}
 G_2(t)
\ge  C\,{\e}, \quad \forall \ t \ge T_1.
\end{align}

 This concludes the proof of Lemma
\ref{F11}.
\end{proof}

\begin{rem}\label{rem3.1}Note that the same conclusions as in Lemmas \ref{F1} and \ref{F11} can be obtained for any positive nonlinearity of the form $F(u,u_t)$ instead of $|u_t|^p+|u|^q$. Indeed,  in the proofs of these lemmas we only used the non negativity of the nonlinearities $|u_t|^p$ and $|u|^q$. Moreover, the  results of Lemmas \ref{F1} and \ref{F11} naturally hold true for a one nonlinearity $|u_t|^p$ or $|u|^q$ as it is the case for  (\ref{T-sys-bis}). 
\end{rem}

\section{Proof of Theorem \ref{blowup}}\label{proof}
\par\quad

This section is aimed to detail the proof of the first main blow-up result which is related to the solution of the Cauchy problem (\ref{T-sys}). To this end,  we will  make use of the lemmas proven in Section \ref{aux} together with a Kato's lemma type.

Following the hypotheses in Theorem \ref{blowup}, we recall that $\mbox{\rm supp}(u)\ \subset\{(x,t)\in\R^N\times[0,\infty): |x|\le \xi(t)+R\}$. 
For $t \in [0,T)$, we set
\begin{equation}
F(t):=\int_{\R^N}u(x,t)dx.
\end{equation}
Recall here the weak formulation \eqref{energysol2}  in which we set a test function $\Phi$  such that
$\Phi\equiv 1$ in $\{(x,s)\in \R^N\times[0,t]:|x|\le \xi(s)+R\}$\footnote{ This choice  is possible thanks to the fact that the energy solution $u$ verifies $\mbox{\rm supp}(u)\ \subset\{(x,t)\in\R^N\times[0,\infty): |x|\le \xi(t)+R\}$.}, and we obtain
\begin{equation}
\label{energysol1}
\begin{array}{r}
\d \d\int_{\R^N}u_t(x,t)dx-\int_{\R^N}u_t(x,0)dx =\d \int_0^t  \int_{\R^N}\left\{|u_t(x,s)|^p+|u(x,s)|^q\right\}dx \,ds.
\end{array}
\end{equation}
Using the definition of $F(t)$, the equation \eqref{energysol1} yields
\begin{equation}
\label{F'0ineq12}
F'(t)= F'(0)+ \int_0^t \int_{\R^N}\left\{|u_t(x,s)|^p+|u(x,s)|^q\right\}dx \,ds.
\end{equation}
Differentiating the above equation in time, we infer that
\begin{equation}
\label{F'0ineq1-bis}
F''(t)= \int_{\R^N}\left\{|u_t(x,t)|^p+|u(x,t)|^q\right\}dx.
\end{equation}
Integrating \eqref{F'0ineq1-bis} twice in time over $(0,t)$  and using the positivity of $ F(0)$ and $F'(0)$, we deduce that
\begin{align}
\label{F'0ineqmmint}
F(t)\geq \int_0^t \int_0^s \int_{\R^N}\left\{|u_t(x,\tau)|^p+|u(x,\tau)|^q\right\} dx \,d\tau\,ds.
\end{align}
Thanks to the  H\"{o}lder's inequality and the estimates \eqref{psi} and \eqref{F2postive}, we have
\begin{equation}\label{4.6}
\begin{array}{rcl}
\d \int_{\R^N}|u_t(x,t)|^pdx &\geq& \d G_2^p(t)\left(\int_{|x|\leq \xi(t)+R}\Big(\psi(x,t)\Big)^{\frac{p}{p-1}}dx\right)^{-(p-1)} \vspace{.2cm}\\  &\geq&  C\rho^{-p}(t)e^{-p\xi(t)}\e^p(\xi(t))^{ -\frac{(N-1)(p-2)}2}, \quad \forall \ t \ge T_1.
\end{array}
\end{equation}
From the expression of $\xi(t)$, given by \eqref{xi}, \eqref{lmabdaK} and \eqref{est-double}, we deduce that
 \begin{equation}\label{pho-est}
 \d \rho(t)e^{\xi(t)} \le C t^{\frac{-m}{2}}, \ \forall \ t \ge T_0/2.
 \end{equation}
Combining \eqref{4.6} and \eqref{pho-est}, we get
\begin{equation}
\d \int_{\R^N}|u_t(x,t)|^pdx \geq C \e^p t^{ -\frac{(m+1)(N-1)(p-2)-mp}2}, \ \forall \ t \ge T_1.\\
\end{equation}
Plugging the above inequality  into \eqref{F'0ineqmmint}, we obtain
\begin{equation}
\begin{aligned}\label{F0first}
F(t)
&\geq C\e^p t^{2 -\frac{(m+1)(N-1)(p-2)-mp}2}, \ \forall \ t \ge T_1.
\end{aligned}
\end{equation}
Otherwise, we have
\begin{align}
\Big(\int_{\R^N}u(x,t)dx\Big)^q\le \int_{|x|\le \xi(t)+R}|u(x,t)|^qdx  \Big(\int_{|x|\le \xi(t)+R}dx\Big)^{q-1}, 
\end{align}
which implies that
\begin{align}\label{f0qsup}
F^q(t)\le  C(1+t)^{N(q-1)(m+1)}  \int_{|x|\le \xi(t)+R}|u(x,t)|^q dx.
\end{align}
 Gathering \eqref{F'0ineq1-bis} and \eqref{f0qsup}, we deduce that
\begin{equation}
\label{F'0ineq2-1}
F''(t) \ge  C\frac{F^q(t)}{(1+t)^{N(q-1)(m+1)}}, \ \forall \ t > 0.
\end{equation}
From   \eqref{F'0ineq12} we have $F'(t)>0$. Then, multiplying \eqref{F'0ineq2-1} by $F'(t)$ and integrating the obtained inequality, we get
\begin{equation}
\label{F25nov4-1}
\Big(F'(t)\Big)^2
\ge C\frac{F^{q+1}(t)}{(1+t)^{N(q-1)(m+1)}} +\left((F'(0))^2-CF^{q+1}(0)\right), \ \forall \ t > 0.
\end{equation}
Since we consider here small initial data, we can easily see that  the last term in the right-hand side of \eqref{F25nov4-1} is positive,  and more precisely this holds for $\e$ small enough. 
Therefore \eqref{F25nov4-1} gives
\begin{equation}
\label{F25nov6}
\frac{F'(t)}{F^{1+\delta}(t)}
\ge C \frac{F^{\frac{q-1}2-\delta}(t)}{(1+t)^{\frac{N(q-1)(m+1)}{2}}}, \ \forall \ t > 0,
\end{equation}
for $\delta>0$ small enough.\\
Integrating the inequality \eqref{F25nov6} on $[t_1,t_2]$, for all $t_2>t_1 \ge T_1$, and using   \eqref{F0first}, we obtain 
\begin{equation}\label{4.14}
\frac1{\delta}\Big (\frac{1}{F^{\delta}(t_1)}-\frac{1}{F^{\delta}(t_2)}\Big)
\ge C  (\e^p )^{\frac{q-1}2-\delta} \int_{t_1}^{t_2} \frac{(1+s)^{(2 -\frac{(m+1)(N-1)(p-2)-mp}2)(\frac{q-1}2-\delta)}}{(1+s)^{\frac{N(q-1)(m+1)}{2}}}ds, \ \forall \ t_2>t_1 \ge T_1.
\end{equation}
Neglecting the second term in the left-hand side of \eqref{4.14} and using the definition of $\Lambda(p, q, N,m)$ (given by \eqref{1.5}) yield
\begin{equation}\label{eq1/F}
\frac{1}{F^{\delta}(t_1)}
\ge C \delta \e^{\frac{p(q-1)}2-p\delta} \int_{t_1}^{t_2} (1+s)^{-\frac{ \Lambda(p, q, N,m)}{4}-\delta\left(2 -\frac{(m+1)(N-1)(p-2)-mp}2\right)}ds.  
\end{equation}
Thanks to the hypothesis \eqref{assump}, we can choose $\delta=\delta_0$ small enough such that $\gamma:=-\frac{ \Lambda(p, q, N,m)}{4}-\delta_0\left(2 -\frac{(m+1)(N-1)(p-2)-mp}2\right)>-1$. Then, the estimate \eqref{eq1/F} implies that
\begin{equation}\label{eq1/F-2}
\frac{1}{F^{\delta_0}(t_1)}
\ge C   \e^{\frac{p(q-1)}2-p\delta_0} \, \left((1+t_2)^{\gamma+1}-(1+t_1)^{\gamma+1}\right), \ \forall \ t_2>t_1 \ge T_1.
\end{equation}
From \eqref{F0first},  we infer that
\begin{equation}\label{eq1/F-3}
\e^{\frac{p(q-1)}2} \, \left((1+t_2)^{\gamma+1}-(1+t_1)^{\gamma+1}\right) 
\le C_3 (1+t_1)^{-\delta_0\left(2 -\frac{(m+1)(N-1)(p-2)-mp}2\right)}, \ \forall \ t_2>t_1 \ge T_1.
\end{equation}
Consequently, we have 
\begin{eqnarray}\label{eq1/F-4}
\e^{\frac{p(q-1)}2} \, (1+t_2)^{\gamma+1} 
&\le& C_3 (1+t_1)^{-\delta_0\left(2 -\frac{(m+1)(N-1)(p-2)-mp}2\right)}\\&&+\e^{\frac{p(q-1)}2} \, (1+t_1)^{\gamma+1}, \ \forall \ t_2>t_1 \ge T_1.\nonumber
\end{eqnarray}
Since $-\frac{ \Lambda(p, q, N,m)}{4}+1>0$, then for all  $\e>0$, we choose $\tilde{T}_2$ such that
\begin{equation}\label{eq1/F-5}
\tilde{T}_2=C_3^{-\frac{4}{4-\Lambda(p, q, N,m)}} \e^{-\frac{2p(q-1)}{4-\Lambda(p, q, N,m)}}.
\end{equation}
Finally, we set $t_1=\max (T_1, \tilde{T}_2)$
and we plug \eqref{eq1/F-5} in \eqref{eq1/F-4} to obtain that
\begin{equation}\label{eq1/F-6}
t_2 \le 2^{\frac{1}{\gamma+1}}(1+t_1) \le C \e^{-\frac{2p(q-1)}{4-\Lambda(p, q, N,m)}}.
\end{equation}

The proof of Theorem \ref{blowup} is now complete.\hfill $\Box$

\section{Proof of Theorem \ref{th_u_t}.}\label{sec-ut}

The purpose of this section is to give the details on the proof of Theorem \ref{th_u_t} which is related to 
the solution of \eqref{T-sys-bis}. For that aim, we use the computations already obtained in Section \ref{aux} .  First, we  note that Lemmas \ref{F1} and \ref{F11} remain true for the solution of \eqref{T-sys-bis} instead of \eqref{T-sys} (see Remark \ref{rem3.1}) since we only used the non negativity of the nonlinear terms and not their types.

In fact, the result in this section is somehow an application (and does not constitute in any case the main objective of this work) of the result obtained for problem \eqref{T-sys} with mixed nonlinearities in Theorem \ref{blowup}. Indeed, we note here that the blow-up result for the system \eqref{T-sys-bis} is also obtained in \cite{Lai2020}; the result there is the same as in Theorem \ref{th_u_t}. However, our approach here is totally different. Moreover, we believe that our method can be used to study the blow-up of the system with generalized Tricomi term. This will be the subject of a forthcoming work.

In order to prove the blow-up result for \eqref{T-sys-bis}  we will use  the estimate \eqref{G2+bis4} (initially proved for \eqref{T-sys}) with omitting the nonlinear term $|u(x,t)|^q$. Hence, the analogous of the estimate \eqref{G2+bis4} reads as follows:
\begin{equation}\label{G2+bis55}
\begin{array}{c}
\d G_2'(t)+\frac{3\Gamma(t)}{4}G_2(t)\ge  \frac{\e \, C_m(f,g) \,t^{m}}{4}+\int_{\R^N}|u_t(x,t)|^p\psi(x,t) dx \vspace{.2cm}\\+\d \frac{t^{m}}{4}\int_0^t \int_{\R^N}|u_t(x,s)|^p\psi(x,s) dx ds, \quad \forall \ t  \ge  \tilde{T}_1.
\end{array}
\end{equation}
Now, we introduce the following functional:
\[
H(t):=
\frac{1}{8}\int_{\tilde{T}_3}^t \int_{\R^N}|u_t(x,s)|^p\psi(x,s)dx ds
+\frac{C_4 \e}{8},
\]
where $C_4=\min(C_m(f,g)/4,8C_{G_2})$ ($C_{G_2}$ is defined in Lemma \ref{F11}) and $\tilde{T}_3>T_1$ is chosen such that $\frac{t^{m}}{4}-\frac{3\Gamma(t)}{32}>0$ for all $t \ge\tilde{T}_3$ (this is possible thanks to \eqref{gamma} and \eqref{lambda'lambda1}).\\
Let
$$\mathcal{F}(t):=G_2(t)-H(t),$$
which verifies
\begin{equation}\label{G2+bis6}
\begin{array}{rcl}
\d \mathcal{F}'(t)+\frac{3\Gamma(t)}{4}\mathcal{F}(t) &\ge& \d \left(\frac{t^{m}}{4}-\frac{3\Gamma(t)}{32}\right)\int_{\tilde{T}_4}^t \int_{\R^N}|u_t(x,s)|^p\psi(x,s)dx ds\vspace{.2cm}\\ &+&  \d \frac{7}{8}\int_{\R^N}|u_t(x,t)|^p\psi(x,t) dx+C_4 \left(t^{m}-\frac{3\Gamma(t)}{32}\right) \e\\
&\ge&0, \quad \forall \ t \ge \tilde{T}_3.
\end{array}
\end{equation}
Multiplying  \eqref{G2+bis6} by $\frac{1}{\rho^{3/2}(t)}$ and integrating over $(\tilde{T}_3,t)$, we obtain
\begin{align}\label{est-G111}
 \mathcal{F}(t)
\ge \mathcal{F}(\tilde{T}_3)\frac{\rho^{3/2}(t)}{\rho^{3/2}(\tilde{T}_3)}, \quad \forall \ t \ge \tilde{T}_3,
\end{align}
where $\rho(t)$ is defined by \eqref{lmabdaK}.\\
Hence, we have $\d \mathcal{F}(\tilde{T}_3)=G_2(\tilde{T}_3)-\frac{C_4 \e}{8} \ge G_2(\tilde{T}_3)-C_{G_2}\e \ge 0$ thanks to Lemma \ref{F11} and the fact that $C_4=\min(C_m(f,g)/4,8C_{G_2}) \le 8C_{G_2}$. \\
Therefore we deduce that 
\begin{equation}
\label{G2-est}
G_2(t)\geq H(t), \ \forall \ t \ge \tilde{T}_3.
\end{equation}
On the other hand, using H\"{o}lder's inequality and the estimates \eqref{psi} and \eqref{F2postive}, we infer that
\begin{equation}
\begin{array}{rcl}
\d \int_{\R^N}|u_t(x,t)|^p\psi(x,t)dx &\geq&\d G_2^p(t)\left(\int_{|x|\leq \xi(t)+R}\psi(x,t)dx\right)^{-(p-1)} \vspace{.2cm}\\ &\geq& C G_2^p(t) \rho^{-(p-1)}(t)e^{-(p-1)\xi(t)}(\xi(t))^{-\frac{(N-1)(p-1)}2}.
\end{array}
\end{equation}
Thanks to \eqref{pho-est}, we obtain  
\begin{equation}
\d \int_{\R^N}|u_t(x,t)|^p\psi(x,t)dx \geq C G_2^p(t) t^{-\frac{\left[(N-1)(m+1)-m\right](p-1)}{2}}, \ \forall \ t \ge \tilde{T}_3.
\end{equation}
Using the above estimate and \eqref{G2-est}, we deduce that
\begin{equation}
\label{inequalityfornonlinearin}
H'(t)\geq C H^p(t)t^{-\frac{\left[(N-1)(m+1)-m\right](p-1)}{2}}, \quad \forall \ t \ge \tilde{T}_3.
\end{equation}
Recall that $H(\tilde{T}_3)=\frac{C_4 \e}{8}>0$, 
we  easily obtain the upper bound of the lifespan estimate as in Theorem \ref{th_u_t} which concludes the proof.

\section{Appendix}\label{app}
The appendix is aimed to give some light on the  property  {\bf (iii)} in Section \ref{aux}, namely \eqref{lambda'lambda1},  of the function $\rho(t)$, the solution of \eqref{lambda}, by showing some elementary results on the behavior of $\rho(t)$ for $t$ large enough. Hence, we first rewrite the expression of  $\rho(t)$ as follows:
\begin{equation}\label{lmabdaK-A}
\rho(t)=\left\{
\begin{array}{ll}\d \a_m t^{\frac{1}{2}}K_{\frac{1}{2m+2}}\left(\xi(t)\right), &\quad \forall \ t>0,\\
1, &\quad \text{at} \ t=0,
 \end{array}
  \right.
\end{equation}
where 
$$K_{\nu}(z)=\int_0^\infty\exp(-z\cosh \zeta)\cosh(\nu \zeta)d\zeta,\ \nu\in \mathbb{R}.$$
Using the fact that
\begin{equation}\label{Knu-pp}
\frac{d}{dz}K_{\nu}(z)=-K_{\nu+1}(z)+\frac{\nu}{z}K_{\nu}(z),
\end{equation}
 we infer that
\begin{equation}\label{lambda'lambda}
\frac{\rho'(t)}{\rho(t)}=\frac{1}{t}-\frac{K_{\frac{2m+3}{2m+2}}(\xi(t))}{K_{\frac{1}{2m+2}}(\xi(t))}t^m, \quad \forall \ t>0.
\end{equation}
From \cite{Gaunt}, we have the following property of the function $K_{\mu}(t)$:
\begin{equation}\label{Kmu}
K_{\nu}(t)=\sqrt{\frac{\pi}{2t}}e^{-t} (1+O(t^{-1}), \quad \text{as} \ t \to \infty.
\end{equation}
Combining \eqref{lambda'lambda} and \eqref{Kmu}, we deduce the property \eqref{lambda'lambda1}.

Finally, we refer the reader to \cite{Erdelyi,Gaunt} for more details about the properties of the function $K_{\nu}(t)$.


\bibliographystyle{plain}

\end{document}